\documentclass[draft]{article}
\usepackage{amsthm}
\usepackage{amsmath}
\usepackage{amssymb}
\usepackage[latin1]{inputenc}
\usepackage{hyperref}

\newtheorem{teo}{Theorem}

\newtheorem{coro}[teo]{Corollary}
\newtheorem{prop}[teo]{Proposition}
\newtheorem{lema}[teo]{Lemma}

\newtheorem*{prob}{Problem}

\newcommand{\KX}{K\langle X \rangle}
\newcommand{\KXk}{K\langle x_1,\dots,x_k\rangle}

\begin{document}

\title{Minimal Varieties and Identities of Relatively Free Algebras}

\author{
Dimas Jos\'e Gon\c{c}alves\thanks{\texttt{dimas@dm.ufscar.br};
partially supported by CAPES and CNPq}
\\
\\
Departamento de Matem\'atica,\\
Universidade Federal de S\~ao Carlos,\\
13565-905 S\~ao Carlos, SP, Brazil
\\
\\
Thiago Castilho de Mello \thanks{\texttt{tcmello@unifesp.br}; partially supported by Fapesp grant No. 2012/16838-0}
\\
\\
Instituto de Ci\^encia e Tecnologia\\
Universidade Federal de S\~ao Paulo\\
12231-280 S\~ao Jos\'e dos Campos, SP, Brazil}

\date{}
\maketitle

\begin{abstract}
Let $K$ be a field of characteristic zero and let $\mathfrak{M}_5$ be the variety of associative algebras over $K$, defined by the identity
$[x_1,x_2][x_3,x_4,x_5]$. It is well-known that such variety is a minimal variety and that is generated by the algebra
\[A=\begin{pmatrix}
E_0 & E\\
0 & E\\
\end{pmatrix},\]
where $E=E_0\oplus E_1$ is the Grassmann algebra.
In this paper, for any positive integer $k$, we describe the polynomial identities of the relatively free algebras of rank $k$ of $\mathfrak{M}_5$,
\[F_k(\mathfrak{M}_5)=\dfrac{\KXk}{\KXk\cap T(\mathfrak{M}_5)}.\]
It turns out that such algebras satisfy the same polynomial identities of some algebras used in the description of the subvarieties of $\mathfrak{M}_5$,
given by Di Vincenzo, Drensky and Nardozza.
\end{abstract}

\section{Introduction}

Let $K$ be a field of characteristic 0, and let $X=\{x_1,x_2,\dots\}$ be an infinite countable set.
We denote by $\KX$ the free associative
unitary algebra freely
generated by $X$.
We say that an associative algebra $A$ is an \emph{algebra with
polynomial identity} (PI-algebra, for short) if there exists a nonzero polynomial $f=f(x_1,\dots,x_n)\in \KX$ such that $f(a_1,\dots, a_n)=0$, for any
elements $a_1$, \dots, $a_n\in A$. In this case, we say that $f$ is a polynomial identity of $A$, or simply that $A$ satisfies $f$.

If $A$ is a PI-algebra
then $T(A)=\{f\in \KX \,|\, f \text{ is an identity of } A\}$ is an ideal of $\KX$ which is invariant under any endomorphism of
the algebra $\KX$. An ideal with this property is called a \emph{T-ideal} or \emph{verbal ideal} of $\KX$. We refer the reader to
\cite{Drensky,GZ2005} for the basic theory of PI-algebras.

The T-ideals play a central role in the theory of PI-algebras, and they are often studied through its equivalent notion of varieties of algebras. If
$\mathcal{F}$ is a subset of $\KX$, the class of all algebras satisfying the identities from $\mathcal{F}$ is called \emph{the variety of (associative)
algebras defined by $\mathcal{F}$} and denoted by $var(\mathcal{F})$.
Given $\mathfrak{V}$ and $\mathfrak{W}$ varieties of algebras, we say that
$\mathfrak{W}$ is a subvariety of $\mathfrak{V}$ if $\mathfrak{W}\subseteq \mathfrak{V}$. If $\mathfrak{V}$ is a variety of algebras, we denote by
$T(\mathfrak{V})$ the set of
 identities satisfied by all algebras in $\mathfrak{V}$.
 Of course $T(\mathfrak{V})$ is a T-ideal of $\KX$; it is called \emph{the T-ideal of $\mathfrak{V}$}.
If $\mathfrak{V}=var(\mathcal{F})$, we say that the elements of $T(\mathfrak{V})$ are \emph{consequences of}
(or \emph{follow from}) the elements of
$\mathcal{F}$.
Note that $T(\mathfrak{V})$ is the smallest T-ideal that contains $\mathcal{F}$. It is the vector space generated by all elements
\[g_0f(g_1,\ldots,g_n)g_{n+1}\]
where $g_0,\ldots,g_{n+1} \in \KX$ and $f(x_1,\ldots,x_n) \in \mathcal{F}$. In this case we denote
\[T(\mathfrak{V})=\langle \mathcal{F} \rangle^T\]
and we say that the T-ideal $\langle \mathcal{F} \rangle^T$ is \emph{generated} by $\mathcal{F}$.

If $\mathfrak{V}$ is a variety and $T(\mathfrak{V})=T(A)$ for some associative algebra $A$, we say that $A$ generates the variety $\mathfrak{V}$. It can be
shown that $\mathfrak{V}$ is generated by the factor algebra
\[F(\mathfrak{V})=\dfrac{\KX}{T(\mathfrak{V})}\]
called the relatively free algebra of $\mathfrak{V}$.

If $\mathfrak{V}$ is a variety, the relatively free algebra of $\mathfrak{V}$ satisfies a universal property: given $A\in \mathfrak{V}$ and an arbitrary
function $\varphi_0:X\longrightarrow A$, there exists a unique homomorphism of algebras $\varphi:F(\mathfrak{V})\longrightarrow A$ extending $\varphi_0$.
One can also define the relatively free algebra of finite rank.
If $k$ is a positive integer, the relatively free algebra of $\mathfrak{V}$ of rank $k$ is
the factor algebra
\[F_k(\mathfrak{V})=\dfrac{\KXk}{T(\mathfrak{V})\cap \KXk},\] which also satisfies a similar universal property. Of course, for any $k$,
$F_k(\mathfrak{V})\in \mathfrak{V}$, but it is not true in general that there exists $k$ such that $\mathfrak{V}$ is generated by $F_k(\mathfrak{V})$. If
such $k$ exists, we say that $\mathfrak{V}$ has (finite) \emph{basic rank} $k$.
Otherwise, $\mathfrak{V}$ has \emph{infinite basic rank}. For example, the
algebra of $n\times n$ matrices over $K$ generates a variety of basic rank 2,
while the variety generated by the infinite dimensional Grassmann algebra,
$E=E_0\oplus E_1$, has infinite basic rank.
Here $E_0$ and $E_1$ are the usual notations for
the even and odd components of $E$.

If $n\in \mathbb{N}$, let $P_n$ be the set of multilinear polynomials of degree $n$ in the variables $x_1$, $x_2$, \dots, $x_n$. Regev \cite{Regev} introduced
some numerical invariants for a variety of algebras. If $\mathfrak{V}$ is a variety the sequence
\[c_n(\mathfrak{V}):=\dim \dfrac{P_n}{P_n\cap
T(\mathfrak{V})}\] is called the codimension sequence of $\mathfrak{V}$. Regev  proved that if $\mathfrak{V}$ is a nontrivial variety, then
$c_n(\mathfrak{V})$ is exponentially bounded. In particular, it is well defined $exp(\mathfrak{V}):=\limsup \sqrt[n]{c_n(\mathfrak{V})}$, called the
exponent of $\mathfrak{V}$.

Amitsur conjectured in the 80's that for any variety, the limit $\displaystyle\lim_{n\to \infty}\sqrt[n]{c_n(\mathfrak{V})}$ exists and is an integer. This
result was only established by Giambruno and Zaicev in the end of the 90's \cite{GZ98, GZ99}.

A description of varieties of algebras with small exponent is known. A variety $\mathfrak{V}$ has exponent 1 if and only if $c_n(\mathfrak{V})$ is
polynomially bounded, which is equivalent to $E\not \in \mathfrak{V}$ and $UT_2(K)\not \in \mathfrak{V}$ \cite{Kemer78, Kemer79}. Here $UT_n(K)$ is the
algebra of $n\times n$ upper-triangular matrices over $K$. Giambruno and Zaicev \cite{GZ2000} listed five algebras with the property that a variety has
exponent $\leq 2$ if and only if it does not contain any of these algebras. The algebras listed by Giambruno and Zaicev are $UT_3(K)$, $M_2(K)$ and the
subalgebras of $M_2(E)$:
\[A=\begin{pmatrix}%
E_0 & E\\
0 & E\\
\end{pmatrix}, \ \ A'=\begin{pmatrix}%
E & E\\
0 & E_0\\
\end{pmatrix} \ \ \mbox{and} \ \ M_{1,1}(E)=\begin{pmatrix}%
E_0 & E_1\\
E_1 & E_0\\
\end{pmatrix}.\]

The first two algebras have finite basic rank while the last three are of infinite basic rank. Each of the above algebras generates a variety with exponent
greater than 2 and any subvariety of it has exponent $\leq 2$. A variety with this property is called a \emph{minimal variety}. It is worth mentioning that
$UT_3(K)$, $A$ and $A'$ have exponent 3, while $M_2(K)$ and $M_{1,1}(E)$ have exponent 4.

This paper deals with the variety $\mathfrak{M}_5$, generated by the algebra $A$ above.
It follows from a theorem of Lewin \cite{Lewin} that
\[T(A)=T(E_0)T(E)=\langle [x_1,x_2][x_3,x_4,x_5] \rangle^T.\]
Here $[x_1,x_2]=x_1x_2-x_2x_1$ and $[x_3,x_4,x_5]=[[x_3,x_4],x_5]$.

One can observe that the polynomial identities of the algebra $A'$ follow from the polynomial identity $[x_1,x_2,x_3][x_4, x_5]$. Hence the relatively free
algebras $F_n(A)$ and $F_n(A')$ are antiisomorphic and this immediately allows one to transfer the results of $F_n(A)$ to the case of $F_n(A')$.

In \cite{DDN}, Di Vincenzo, Drensky and Nardozza give a decomposition of the
$S_n$-module of proper multilinear polynomials modulo the identities of $A$, $\Gamma_n(\mathfrak{M}_5)$, as a sum of its irreducible components.

In this paper, we use the decomposition of $\Gamma_n(\mathfrak{M}_5)$ to prove our main theorem
(Theorem \ref{teoremaprincipal}) which exhibits for each $n$ a minimal set of generators of the identities of
the relatively free algebra of rank $n$ of $\mathfrak{M}_5$. In general, for an arbitrary variety
$\mathfrak{V}$ the question of identities for $F_n(\mathfrak{V})$ has been answered in very few cases.
For example if $\mathfrak{E}$ is the variety generated by the Grassmann algebra, the identities
of $F_n(\mathfrak{E})$ equals the identities of the finite dimensional Grassmann algebra
$E^{(n)}$, and the identities of $F_2(M_{1,1})$ have been described by Koshlukov and the second author in \cite{KdM}.
Last, we exhibit for each $n\geq 1$ a finite dimensional algebra $A^{(n)}$ PI-equivalent to $F_n(\mathfrak{M}_5)$, i.e., $A^{(n)}$ satisfies the same polynomial identities of $F_n(\mathfrak{M}_5)$.

\section{Preliminaries}

Let $K$ be a field of characteristic zero. We denote by $\KX$ the free associative unitary $K$-algebra freely generated by the set $X=\{x_1,x_2,\dots\}$. From
now on (unless otherwise stated) the word \emph{algebra} means an associative and unitary algebra over $K$.
For any PI-algebra $R$ we denote by $T(R)$ its ideal of polynomial identities and for any $n\geq 2$ its relatively free algebra of rank $n$ is defined as
\[F_n(R)=\frac{K\langle x_1,\ldots, x_n\rangle}{K\langle x_1,\ldots,x_n\rangle \cap
T(R)}.\]

If $V$ is a $K$-vector space with a basis $\{e_1, e_2, \dots\}$, the \emph{Grassmann algebra} (or exterior algebra) of $V$, $E=E(V)$ is the vector space with a basis
consisting of 1 and all products $e_{i_1}e_{i_2}\cdots e_{i_k}$, $i_1<i_2<\cdots<i_k$, $k\ge 1$. The multiplication in $E$ is induced by the
anticommutative law for the $e_i$'s, i.e., by $e_ie_j=-e_je_i$ for all $i$ and $j$. In a similar way, one defines the finite dimensional Grassmann algebra $E^{(n)}$, of a finite dimensional vector space with basis $\{e_1, \dots, e_n\}$.

The following result is well-known. See \cite{KrRe} for more details.

\begin{teo}
The T-ideal of the polynomial identities of $E$ is generated by
\[[x_1,x_2,x_3].\]
\end{teo}

One of the facts used in this paper is that the polynomials
\[[x_1,x_2][x_2,x_3] \ \
\mbox{and} \ \ [x_1,x_2][x_3,x_4]+[x_1,x_3][x_2,x_4]\]
follow from the triple commutator. For more details, see
\cite[Lemma 5.1.1]{Drensky}.

The following is a well-known fact (see \cite[Exercise 5.1.3]{Drensky}).
\begin{teo}
Let $k\geq 1$ be an integer.
The T-ideal of the polynomial identities of $E^{(2k)}$ is generated by
\[[x_1,x_2,x_3] \ \ \mbox{and} \ \  [x_1,x_2]\cdots [x_{2k+1},x_{2k+2}].\]
Moreover,
$T(E^{(2k)})=T(E^{(2k+1)})$.
\end{teo}

Given a PI-algebra $R$ and polynomials $f$, $g\in \KX$, we denote \[f \ \equiv ^R \  g\] if $(f-g) \in T(R)$, i.e., if the equality
\[f+T(R)=g+T(R)\]
holds in  $\KX/T(R)$.

The Grassmann algebra is $\mathbb{Z}_2$-graded: one can easily verify that $E=E_0\oplus E_1$
where the vector
subspace $E_j$ is spanned by all
elements
$e_{i_1}e_{i_2}\cdots e_{i_k}$ from its basis
such that $k\equiv j\pmod{2}$.
Here $E_aE_b\subseteq E_{a+b}$
where $a+b$ is considered module 2. Recall that $E_0$ is just the
centre of $E$.

From this decomposition, define the matrix algebra \[A=\left(
                                                        \begin{array}{cc}
                                                          E_0 & E \\
                                                          0 & E \\
                                                        \end{array}
                                                      \right) \ .\]

The T-ideal generated by the polynomial $[x_1, x_2][x_3, x_4, x_5]$ was studied in the paper by Stoyanova-Venkova \cite{SV}. The fact that this T-ideal
coincides with the ideal $T(A)$ is an easy consequence of the theorem of Lewin \cite{Lewin} and the result of Krakowski and Regev \cite{KrRe} about the
polynomial identities of the Grassmann algebra.

\begin{teo}
The T-ideal of the polynomial identities of $A$ is generated by
\[[x_1,x_2][x_3,x_4,x_5].\]
\end{teo}

The so-called higher commutators are defined inductively by $[x_1,x_2]:=x_1x_2-x_2x_1$ and $[x_1,\dots,x_{n}]:=[[x_1,\dots,x_{n-1}],x_{n}]$. The latter is
called a commutator of length $n$. Specht introduced a special type of polynomials. A polynomial is called \emph{proper} if it is a linear combination of
products of commutators. The importance of such polynomials lies in the fact that the identities of an algebra with unit follows from its proper ones
(see \cite[Proposition 4.3.3]{Drensky}). We denote by $P_n$ the vector space of all multilinear polynomials of degree $n$ in $K\langle x_1,\dots,x_n\rangle$ and
by $\Gamma_n$ its subspace spanned by proper multilinear polynomials. If $\mathfrak{M}$ is a variety of algebras,
denote
\[P_n(\mathfrak{M})=\dfrac{P_n}{P_n\cap T(\mathfrak{M})} \ \ \mbox{and} \ \
 \Gamma_n(\mathfrak{M})=\dfrac{\Gamma_n}{\Gamma_n \cap T(\mathfrak{M})} \ .\] It is well-known
that $P_n(\mathfrak{M})$ and $\Gamma_n(\mathfrak{M})$ are modules over the symmetric group $S_n$ .  The main tool to study $\Gamma_n(\mathfrak{M})$ and
$P_n(\mathfrak{M})$ is provided by the representation theory of $S_n$. We refer the reader to \cite[Chapter 12]{Drensky} for all necessary information
concerning the representation of such group and their applications to PI-theory.

We now define a list of polynomials that will be important in the description of the $S_n$-module structure of $\Gamma_n(\mathfrak{M})$. If $m$ and $i$ are
``appropriate" positive integers, we define:

\begin{enumerate}
 \item $f_m^{(1)}=[x_2,x_1,\ldots,x_1]$.

 \item $f_{m,i}^{(2)}=[x_2,x_1,\dots,x_1][x_1,x_2]\cdots [x_{i-1},x_{i}]$.

 \item $f_{m,i}^{(3)}=\displaystyle \sum_{\sigma\in S_{i}} (-1)^{\sigma} [x_{\sigma(1)},x_1,\dots,x_1]
 [x_{\sigma(2)},x_{\sigma(3)}]\ldots [x_{\sigma(i-1)},x_{\sigma(i)}]$.

 \item $f_{m,i}^{(4)}=\displaystyle\sum_{\sigma\in S_i} (-1)^{\sigma}
 [x_2,x_1,x_{\sigma(1)},x_1,\dots,x_1][x_{\sigma(2)},x_{\sigma(3)}]\cdots
[x_{\sigma(i-1)},x_{\sigma(i)}]$.

 \item $f_{m,i}^{(5)}=\displaystyle\sum_{\sigma\in S_i} (-1)^{\sigma}
 [x_{\sigma(1)},x_{\sigma(2)},x_1,\dots,x_1][x_{\sigma(3)},x_{\sigma(4)}]\cdots
[x_{\sigma(i-1)},x_{\sigma(i)}]$.
 \end{enumerate}

\noindent In the above definition, the integer $m$ stands for the degree of the polynomial. The values for $m$ and $i$ are presented in the table below:

\[
 \begin{array}{|c|c|c|c|}
\hline
 \mbox{Polynomial} & m\geq & i\geq & \mbox{$i \mod 2$} \\
 \hline
 1.&2& &\\
 \hline
 2.&4& 2 & \mbox{0}\\
 \hline
 3.&4& 3& \mbox{1}\\
 \hline
 4.&5& 3 & \mbox{1}\\
 \hline
 5.&4& 4& \mbox{0}\\
 \hline
 \end{array}
\]

Before proceed to the description of $\Gamma_m(A)$, we define some polynomials. Let $u_m^{(1)}$ be the complete multilinearization of $f_m^{(1)}$ and
$u_{m,i}^{(j)}$ be the complete multilinearization of $f_{m,i}^{(j)}$. We recall that for an algebra over a field of characteristic zero, a polynomial $f$
is an identity if and only if its complete multilinearization is also an identity. One can show that $u_m^{(1)}$ and $u_{m,i}^{(j)}$ generate irreducible
$S_m$-modules pairwise non-isomorphic, corresponding to the partitions $(m-1, 1)$ for $u_m^{(1)}$, $(m- i, 2, 1^{i-2})$ for $u_{m,i}^{(2)}$ and
$u^{(4)}_{m,i}$, and $(m-i+1,1^{i-1})$ for $u_{m,i}^{(3)}$ and $u^{(5)}_{m,i}$. The above fact and the next theorem can be found in \cite{SV} and in
\cite{DDN}.

\begin{teo}\label{seqcarprdea}
The
$S_m$-module \ $\Gamma_m(A)$ \ decomposes as a direct sum of irreducible $S_m$-modules pairwise non-isomorphic in the following way:

\vspace{0.15cm}

\emph{1}. $\Gamma_2(A)=(KS_2) \cdot u_2^{(1)}$.

\vspace{0.15cm}

 \emph{2}. $\Gamma_3(A)=(KS_3)\cdot u_3^{(1)}$.

 \vspace{0.15cm}

 \emph{3}. $\Gamma_4(A)=(KS_4)\cdot u_4^{(1)}\oplus (KS_4) \cdot u_{4,2}^{(2)}\oplus
 (KS_4) \cdot u_{4,3}^{(3)}\oplus  (KS_4) \cdot u_{4,4}^{(5)}$.

 \vspace{0.15cm}

 \emph{4}. If  $m\geq 5$,  $\Gamma_m(A)$ decomposes as
 \[(KS_m)\cdot u_m^{(1)} \bigoplus_{\substack{i=2 \\ i\text{ even}}}^{m-2} (KS_m) \cdot u_{m,i}^{(2)}
\bigoplus_{\substack{i=3 \\ i\text{ odd}}}^{m-1} (KS_m)\cdot u_{m,i}^{(3)}
 \bigoplus_{\substack{i=3 \\ i\text{ odd}}}^{m-2} (KS_m)\cdot
u_{m,i}^{(4)}  \bigoplus_{\substack{i=4 \\ i\text{ even}}}^{m} (KS_m)\cdot u_{m,i}^{(5)}\] Moreover, the proper cocharacter sequence of $A$ is

$\chi_2(A)=\chi(1^2)$, $\chi_3(A)=\chi(2,1)$, $\chi_4(A)=\chi(3,1)+\chi(2,1^2)+\chi(1^4)+\chi(2^2)$ and if $m\geq 5$, we have

$\chi_m(A)=\sum_{i=2}^{m-1}\chi(m-i+1,1^{i-1})+\sum_{i=2}^{m-2}\chi(m-i,2,1^{i-2})$, if $m$ is odd and

$\chi_m(A)=\sum_{i=2}^{m}\chi(m-i+1,1^{i-1})+\sum_{i=2}^{m-2}\chi(m-i,2,1^{i-2})$, if $m$ is even.
\end{teo}

\section{The polynomial identities of $F_m(A)$}

In this section we shall describe the polynomial identities of the relatively free algebras of finite rank of the variety generated by the algebra $A$.
Since $F_n(A)$ is a subalgebra of $F(A)$, $\Gamma_m(F_n(A))$ is a homomorphic image of $\Gamma_m(A)$. As a consequence, to describe the $S_n$-module
structure of $\Gamma_m(F_n(A))$, it is enough to verify which of the generators of $\Gamma_m(A)$ are identities of $F_n(A)$.

The main result of this paper is the following:

\begin{teo}\label{teoremaprincipal}
Let $n\geq 1$ be an even number. The T-ideal of the polynomial identities of $F_n(A)$ is generated by
\[[x_1,x_2][x_3,x_4,x_5] \  \
and \  \ [x_1,x_2][x_{3},x_{4}]\cdots [x_{n+3},x_{n+4}].\]
Moreover, \[T(F_n(A))=T(F_{n+1}(A)).\]
\end{teo}

In order to prove it, we present
a series of results.

\begin{lema}\label{trocandoasposicoes} Let $n\geq 1$. If $f$ is defined as
\[f=u_0[v_1,v_2]u_1[w_{\sigma(1)},w_{\sigma(2)}]u_2\cdots
u_{n}[w_{\sigma(2n-1)},w_{\sigma(2n)}]u_{n+1},\] with $u_i,v_i,w_i\in K\langle X \rangle$ for all $i$ and
 $\sigma \in S_{2n}$, then
\[f \ \equiv^A  \  (-1)^{\sigma}u_0[v_1,v_2][w_1,w_2]\cdots [w_{2n-1},w_{2n}]u_1 \cdots u_{n+1}.\]
\end{lema}

\begin{proof}
After using the identity
\[c[a,b]=[a,b]c-[a,b,c],\]
 $n$ times, we obtain
\begin{equation}\label{igualda}
f \ \equiv^A  \ u_0[v_1,v_2][w_{\sigma(1)},w_{\sigma(2)}]\cdots [w_{\sigma(2n-1)},w_{\sigma(2n)}]u_1\cdots u_{n+1}.
\end{equation}
Since
\[[x_1,x_2][x_3,x_4] \ \equiv^E \ -[x_1,x_3][x_2,x_4],\]
the identity
\[[w_{\sigma(1)},w_{\sigma(2)}]\cdots
[w_{\sigma(2n-1)},w_{\sigma(2n)}] \ \equiv^E \ (-1)^{\sigma}[w_1,w_2]\cdots [w_{2n-1},w_{2n}]
\] holds.
The above identity and (\ref{igualda}) imply
\[f \ \equiv^A \  (-1)^{\sigma}u_0[v_1,v_2][w_1,w_2]\cdots [w_{2n-1},w_{2n}]u_1 \cdots u_{n+1}.\]
\end{proof}

\begin{prop}\label{segunadaidentdefna}
If $n\geq 2$ is even, the polynomial
\[f=[x_1,x_2][x_3,x_4]\cdots [x_{n+3},x_{n+4}]\]
is a polynomial identity for the algebras $F_n(A)$ and $F_{n+1}(A)$.
\end{prop}

\begin{proof}
If $g\in K\langle x_1,\ldots, x_n \rangle$, denote by $\overline{g}$ the element of $F_n(A)$ given by
\[\overline{g}=g+ K\langle x_1,\ldots, x_n \rangle \cap T(A).\]
Let $m_1,\ldots,m_{n+4}$ be monomials in  $K\langle x_1,\ldots, x_n \rangle$. Since $f$ is multilinear, in order to show that $f$ is a polynomial identity
for $F_n(A)$, it is enough to prove that
\[
f(\overline{m_1},\ldots,\overline{m_{n+4}})=\overline{0},
\]
which is equivalent to prove that
\begin{equation}\label{id001}
f(m_1,\ldots,m_{n+4}) \ \equiv^A \ 0.
\end{equation}
Using the identity
\begin{equation}\label{rel1}
[a,bc]=b[a,c]+[a,b]c
\end{equation} we obtain that
\[f(m_1,\ldots,m_{n+4})\]
is a linear combination of elements of type
\[h=u_0[x_{i_1},x_{i_2}]u_1[x_{i_3},x_{i_4}]u_2 \ldots u_{t-1}[x_{i_{n+3}},x_{i_{n+4}}]u_t,\]
where $u_0,\ldots,u_t$ are monomials in $K\langle x_1,\ldots, x_n \rangle$
and $t=(n+4)/2$. By Lemma \ref{trocandoasposicoes} we have
\[h \ \equiv^A \
\pm u_0[x_{i_1},x_{i_2}][x_{j_3},x_{j_4}] \cdots [x_{j_{n+3}},x_{j_{n+4}}]u_1\cdots u_t,\] where $1\leq j_3 \leq j_4 \leq \ldots \leq j_{n+3} \leq j_{n+4}
\leq n$. Since
\[[x_1,x_2][x_2,x_3] \ \equiv^E \ 0\]
and at least two $x_{j_u}$ among $x_{j_3},x_{j_4}, \ldots, x_{j_{n+3}},x_{j_{n+4}}$ are repeated, it follows that
\[h \ \equiv^A \ 0.\]
Hence $h\in T(A)$ and the relation (\ref{id001}) is proved.

Arguing similarly one obtains that
\[[x_1,x_2][x_3,x_4]\cdots [x_{n+3},x_{n+4}]\]
is a polynomial identity for $F_{n+1}(A)$.
\end{proof}

\begin{lema}\label{identfnaea}
$f\in K\langle x_1,\ldots,x_n\rangle$ is a polynomial identity for $F_n(A)$ if and only if $f$ is a polynomial identity for $A$.
\end{lema}

\begin{proof}
If $g\in K\langle x_1,\ldots, x_n \rangle$, denote by $\overline{g}$ the element
\[\overline{g}=g+ K\langle x_1,\ldots, x_n \rangle \cap T(A)\in F_n(A).\]
It is clear from the definition of $F_n(A)$ that $T(A)\subseteq T(F_n(A))$.

For the converse, suppose that $f(x_1,\ldots,x_n)$ is a polynomial identity for $F_n(A)$. In particular,
\[\overline{0}=f(\overline{x_1},\ldots,\overline{x_n})=\overline{f(x_1,\ldots,x_n)},\]
and $f\in T(A)$.
\end{proof}

Denote by $e_{ij}$ the element of $A$ whose $(i,j)$ entry is 1 and the remaining entries are $0$.
We recall that the generators of $E$ are denoted by
$e_1, \ e_2, \ \ldots$ \ .

The following two lemmas follow from the fact that $f_m^{(1)}$ and $f_{m,i}^{(2)}$ are not polynomial identities for $A$ and from Lemma \ref{identfnaea}.
\begin{lema}\label{fm1}
If $n\geq 2$, the polynomial
\[f_m^{(1)}=[x_2,x_1,\ldots,x_1]\]
is not a polynomial identity for $F_n(A)$.
\end{lema}

\begin{proof}
Since
\[f_m^{(1)}(e_{22},e_{12})=e_{12},\]
$f_m^{(1)}$ is not a polynomial identity for $A$. The result now follows from Lemma \ref{identfnaea}.
\end{proof}

\begin{lema}\label{fm2}
 If $n\geq 2$ is even, the polynomial
 \[f_{m,n}^{(2)}=[x_2,x_1,\dots,x_1][x_1,x_2]\cdots [x_{n-1},x_{n}]\]
is not a polynomial identity for $F_n(A)$.
\end{lema}

\begin{proof}
We first consider $n=2$. In
this case,
 \[f_{m,2}^{(2)}=[x_2,x_1,\dots,x_1][x_1,x_2].\]
 By substituting $x_1=(1+e_1)e_{22}$ and $x_2=e_{12}+e_2e_{22}$ we obtain
 \[f_{m,2}^{(2)}((1+e_1)e_{22},e_{12}+e_2e_{22})=2e_1e_2e_{12}.
\]

Let us now look at $f_{m,n}^{(2)}$ for an arbitrary even number $n$. Since
\[f_{m,n}^{(2)}=f_{m,2}^{(2)}[x_3,x_4]\cdots [x_{n-1},x_n],\]
direct calculations show that
\[f_{m,n}^{(2)}((1+e_1)e_{22},e_{12}+e_2e_{22},e_3e_{22},\ldots,e_ne_{22})=2^{n/2}e_1e_2e_3\cdots e_ne_{12}.\]
Hence, $f_{m,n}^{(2)}$ is not a polynomial identity for $F_n(A)$.
\end{proof}

\begin{lema}\label{lemfmn3}
 If $n\geq 2$ is even, the polynomial
 \[f_{m,n+1}^{(3)}=\sum_{\sigma\in S_{n+1}} (-1)^{\sigma} [x_{\sigma(1)},x_1,\dots,x_1]
 [x_{\sigma(2)},x_{\sigma(3)}]\ldots [x_{\sigma(n)},x_{\sigma(n+1)}]\]
is not a polynomial identity for $F_n(A)$.
\end{lema}

\begin{proof}
Suppose that $f_{m,n+1}^{(3)}$ is a polynomial identity for $F_n(A)$. By Lemma \ref{trocandoasposicoes}, we have
\[f_{m,n+1}^{(3)} \ \equiv^A \ g,\]
where
\[g=(n!) \sum_{i=2}^{n+1} (-1)^{i+1} [x_i,x_1,\dots,x_1]
 [x_1,x_2][x_3,x_4]\ldots \widehat{x_i} \ldots [x_n,x_{n+1}].\]
Here $\widehat{x_i}$ means that the variable $x_i$ occurs only in the first commutator, i.e., in $[x_i,x_1,\dots,x_1]$. Since $T(A) \subset T(F_n(A))$, it
follows that $g$ is also a polynomial identity for $F_n(A)$.
In particular,
\[g'=(1/n!)g(x_1,x_2,\ldots,x_n,x_1x_2) \in T(F_n(A)).\]
We have
\begin{eqnarray*}
 g'&\equiv^A& -[x_2,x_1,\dots,x_1]
 [x_1,x_3][x_4,x_5]\cdots  [x_n,x_1x_2] \\
&& +[x_1x_2,x_1,\dots,x_1]
 [x_1,x_2][x_3,x_4]\cdots  [x_{n-1},x_{n}]\\
&\equiv^A& -[x_2,x_1,\dots,x_1]
 [x_1,x_3][x_4,x_5]\cdots  x_1[x_n,x_2] \\
& & +x_1[x_2,x_1,\dots,x_1]
 [x_1,x_2][x_3,x_4]\cdots  [x_{n-1},x_{n}]\\
&\equiv^A& -[x_2,x_1,\dots,x_1]x_1
 [x_1,x_3][x_4,x_5]\cdots  [x_n,x_2] \\
& & +x_1[x_2,x_1,\dots,x_1]
 [x_1,x_2][x_3,x_4]\cdots  [x_{n-1},x_{n}]\\
&\equiv^A& -[x_2,x_1,\dots,x_1]x_1
 [x_1,x_2][x_3,x_4]\cdots  [x_{n-1},x_{n}] \\
&& +x_1[x_2,x_1,\dots,x_1]
 [x_1,x_2][x_3,x_4]\cdots  [x_{n-1},x_{n}]\\
&\equiv^A& -[x_2,x_1,\dots,x_1,x_1]
 [x_1,x_2][x_3,x_4]\cdots  [x_{n-1},x_{n}] \\
&\equiv^A& - f_{m,n}^{(2)}.
 \end{eqnarray*}
Thus $f_{m,n}^{(2)}$ is also a polynomial identity for $F_n(A)$. This contradicts Lemma \ref{fm2}.
\end{proof}

\begin{lema}\label{fm4}
If $n\geq 2$ is even, the polynomial
\[f_{m,n+1}^{(4)}=\sum_{\sigma\in S_{n+1}} (-1)^{\sigma}
[x_2,x_1,x_{\sigma(1)},x_1,\dots,x_1][x_{\sigma(2)},x_{\sigma(3)}]\ldots [x_{\sigma(n)},x_{\sigma(n+1)}]\] is not a polynomial identity for $F_n(A)$.
\end{lema}

\begin{proof}
Suppose that $f_{m,n+1}^{(4)}$ is an identity for $F_n(A)$. Since the polynomial $f_{m,n+1}^{(4)}$ is equivalent modulo $T(A)$ to
\[g=(n!) \sum_{i=1}^{n+1} (-1)^{i+1} [x_2,x_1,x_i,x_1,\dots,x_1]
 [x_1,x_2][x_3,x_4]\ldots \widehat{x_i} \ldots [x_n,x_{n+1}],\]
we obtain that $g$ is also an identity for $F_n(A)$. In particular,
\[h=(1/n!) g(x_1,x_2,\ldots,x_n,x_1^2) \in T(F_n(A))\]
and then
\begin{eqnarray*}
h&\equiv^A&+
[x_2,x_1,x_1,x_1,\ldots,x_1][x_2,x_3][x_4,x_5]\ldots [x_n,x_1x_1]+\\
&&+[x_2,x_1,x_1x_1,x_1,\dots,x_1][x_1,x_2][x_3,x_4]\ldots [x_{n-1},x_n]\\
&\equiv^A&+2[x_2,x_1,x_1,x_1,\ldots,x_1][x_2,x_3][x_4,x_5]\ldots [x_n,x_1]x_1+\\
&&+[x_2,x_1,x_1x_1,x_1,\dots,x_1][x_1,x_2][x_3,x_4]\ldots [x_{n-1},x_n]\\
&\equiv^A&-2[x_2,x_1,x_1,x_1,\ldots,x_1][x_1,x_2][x_3,x_4]\ldots [x_{n-1},x_n]x_1+\\
&&+x_1[x_2,x_1,x_1,x_1,\dots,x_1][x_1,x_2][x_3,x_4]\ldots [x_{n-1},x_n]+\\
&&+  [x_2,x_1,x_1,x_1,\dots,x_1][x_1,x_2][x_3,x_4]\ldots [x_{n-1},x_n]x_1\\
&\equiv^A&-[[x_2,x_1,x_1,x_1,\ldots,x_1][x_1,x_2][x_3,x_4]\ldots [x_{n-1},x_n],x_1]\\
&\equiv^A&-[x_2,x_1,x_1,x_1,\ldots,x_1,x_1][x_1,x_2][x_3,x_4]\ldots [x_{n-1},x_n].\\
\end{eqnarray*}
This contradicts Lemma \ref{fm2}.
\end{proof}

\begin{lema}\label{fm5}
If $n\geq 2$ is even, the polynomial
 \[
 f_{m,n+2}^{(5)}=\sum_{\sigma\in S_{n+2}} (-1)^{\sigma} [x_{\sigma(1)},x_{\sigma(2)},x_1,\dots,x_1]
 [x_{\sigma(3)},x_{\sigma(4)}]\cdots
[x_{\sigma(n+1)},x_{\sigma(n+2)}]
 \]
is not a polynomial identity for $F_n(A)$.
\end{lema}

\begin{proof}
Suppose that $f_{m,n+2}^{(5)}$ is an identity for $F_n(A)$ and define
\[g=(2/n!)f_{m,n+2}^{(5)}.\]
Then we have
\begin{eqnarray*}
g&\equiv^A&\sum_{i<j} (-1)^{j-i+1} [x_i,x_j,x_1,\dots,x_1][x_1,x_2] \ldots \widehat{x_i} \ldots
 \widehat{x_j} \ldots [x_{n+1},x_{n+2}].
\end{eqnarray*}
Define $h=g(x_1,\ldots,x_{n+1},x_1^2)$. Note that $h$ is a polynomial identity for $F_n(A)$. We have
\begin{eqnarray*}
h&\equiv^A&+\sum_{j=2}^{n+1} (-1)^{j} [x_1,x_j,x_1,\dots,x_1][x_2,x_3] \ldots
 \widehat{x_j} \ldots [x_{n+1},x_1x_1]+\\
&&+\sum_{i=2}^{n+1} (-1)^{-i+1} [x_i,x_1x_1,x_1,\dots,x_1][x_1,x_2] \ldots \widehat{x_i} \ldots [x_{n},x_{n+1}]\\
&\equiv^A&+2\sum_{j=2}^{n+1} (-1)^{j} [x_1,x_j,x_1,\dots,x_1][x_2,x_3] \ldots
 \widehat{x_j} \ldots [x_{n+1},x_1]x_1+\\
&&+\sum_{i=2}^{n+1} (-1)^{-i+1} x_1[x_i,x_1,x_1,\dots,x_1][x_1,x_2] \ldots \widehat{x_i} \ldots [x_{n},x_{n+1}]+\\
&&+\sum_{i=2}^{n+1} (-1)^{-i+1} [x_i,x_1,x_1,\dots,x_1]x_1[x_1,x_2] \ldots \widehat{x_i} \ldots [x_{n},x_{n+1}].\\
 \end{eqnarray*}
And also
\begin{eqnarray*}
h&\equiv^A&+2\sum_{j=2}^{n+1} (-1)^{j} [x_1,x_j,x_1,\dots,x_1][x_2,x_3] \ldots
 \widehat{x_j} \ldots [x_{n+1},x_1]x_1+\\
&&+2\sum_{i=2}^{n+1} (-1)^{-i+1} [x_i,x_1,x_1,\dots,x_1][x_1,x_2] \ldots \widehat{x_i} \ldots [x_{n},x_{n+1}]x_1+\\
&&+\sum_{i=2}^{n+1} (-1)^{-i+1} [x_1,[x_i,x_1,x_1,\dots,x_1]][x_1,x_2] \ldots \widehat{x_i} \ldots [x_{n},x_{n+1}].\\
&\equiv^A&+\sum_{i=2}^{n+1} (-1)^{-i+1} [x_1,[x_i,x_1,x_1,\dots,x_1]][x_1,x_2] \ldots \widehat{x_i} \ldots [x_{n},x_{n+1}].\\
&\equiv^A&+\sum_{i=2}^{n+1} (-1)^{i} [x_i,x_1,x_1,\dots,x_1,x_1][x_1,x_2] \ldots \widehat{x_i} \ldots [x_{n},x_{n+1}].\\
&\equiv^A&-(1/n!)f_{m,n+1}^{(3)}.
\end{eqnarray*}
Hence $f_{m,n+1}^{(3)}$ would be an identity for $F_n(A)$. A contradiction to Lemma \ref{lemfmn3}.
\end{proof}

We are now in conditions to prove {\bf Theorem \ref{teoremaprincipal}}.

\begin{proof}[Theorem \ref{teoremaprincipal} ]
If $n\geq 2$ is even, we will decompose $\Gamma_m(F_n(A))$ as a sum of irreducible $S_m$-modules. Since $T(A)\subset T(F_n(A))$, $\Gamma_m(F_n(A))$ is a
homomorphic image of $\Gamma_m(A)$, and its decomposition is given by the same decomposition of $\Gamma_m(A)$ except for the components generated by
polynomials which are identities of $F_n(A)$. We now determine which of these generators are identities of $F_n(A)$.

Observe that $u_m^{(1)}\in T(F_n(A))$ if and only if $f_m^{(1)} \in T(F_n(A))$, since the former is a multilinearization of the latter. The same holds for
$u_{m,i}^{(j)}$, i.e., $u_{m,i}^{(j)} \in T(F_n(A))$ if and only if $f_{m,i}^{(j)} \in T(F_n(A))$.

\begin{enumerate}
 \item Analyzing  $(KS_m)\cdot u_m^{(1)}$.

By Lemma \ref{fm1}, $(KS_m)\cdot u_m^{(1)}$ is a nonzero component of $\Gamma_m(F_n(A))$.

\item Analyzing $(KS_m) \cdot u_{m,i}^{(2)}$.

a) If $i \leq n$, then $(KS_m) \cdot u_{m,i}^{(2)}$ is a nonzero component of $\Gamma_m(F_n(A))$. Indeed, if for some $i \leq n$, $u_{m,i}^{(2)} \in
T(F_n(A))$ then $f_{m,i}^{(2)} \in T(F_n(A))$. Hence
\[f_{m,n}^{(2)}=f_{m,i}^{(2)}[x_{i+1},x_{i+2}]\ldots [x_{n-1},x_{n}] \in T(F_n(A)).\]
And this contradicts Lemma \ref{fm2}.

b) If $i \geq n+2$,
\[f_{m,i}^{(2)} \in \langle [x_1,x_2]\ldots [x_{n+3},x_{n+4}] \rangle ^T,\] and it follows from Proposition \ref{segunadaidentdefna} that
$f_{m,i}^{(2)} \in T(F_n(A))$. Then the $S_m$-module $(KS_m) \cdot u_{m,i}^{(2)}$ vanishes in $\Gamma_m(F_n(A))$.

\item Analyzing $(KS_m) \cdot u_{m,i}^{(3)}$.

a) If $i \leq n+1$, then $(KS_m) \cdot u_{m,i}^{(3)}$ is a nonzero component of $\Gamma_m(F_n(A))$. Indeed, if for some $i \leq n+1$, $u_{m,i}^{(3)} \in
T(F_n(A))$ then $f_{m,i}^{(3)} \in T(F_n(A))$. Hence
\[f_{m,n+1}^{(3)} \in \langle f_{m,i}^{(3)} \rangle^T \subset T(F_n(A)).\]
And this contradicts Lemma \ref{lemfmn3}.

b) If $i \geq n+3$,
\[f_{m,i}^{(3)} \in \langle [x_1,x_2]\ldots [x_{n+3},x_{n+4}] \rangle ^T,\] and it follows from Proposition \ref{segunadaidentdefna} that $f_{m,i}^{(3)} \in T(F_n(A))$.
Then the $S_m$-module $(KS_m) \cdot u_{m,i}^{(3)}$ vanishes in  $\Gamma_m(F_n(A))$.

\item Analyzing $(KS_m) \cdot u_{m,i}^{(4)}$.

a) If $i \leq n+1$, then $(KS_m) \cdot u_{m,i}^{(4)}$ is a nonzero component of $\Gamma_m(F_n(A))$. To see
this, use the same argument used in item a) of the above case and Lemma \ref{fm4}.

b) If $i \geq n+3$, then $(KS_m) \cdot u_{m,i}^{(4)}$ vanishes in $\Gamma_m(F_n(A))$. To see
this, use the same argument used in item b) of the above case

\item Analyzing $(KS_m) \cdot u_{m,i}^{(5)}$.

a) If $i \leq n+2$, then $(KS_m) \cdot u_{m,i}^{(5)}$ is a nonzero component of $\Gamma_m(F_n(A))$. To see
this, use the same argument used in item a) of the above case and Lemma \ref{fm5}.

b) If $i \geq n+4$, then $(KS_m) \cdot u_{m,i}^{(5)}$ vanishes in $\Gamma_m(F_n(A))$.
To see this, use the same argument used in item b) of the above case

\end{enumerate}

Now let $I$ denote the T-ideal generated by
\[[x_1,x_2][x_3,x_4,x_5] \  \ \mbox{and} \  \ [x_1,x_2][x_{3},x_{4}]\cdots [x_{n+3},x_{n+4}].\]
We will find the decomposition of
\[\Gamma_m(I)=\frac{\Gamma_m}{\Gamma_m \cap I}\]
in irreducible $S_m$-modules. Since
\[T(A) \subseteq I \subseteq T(F_n(A))\]
we have two important remarks:
\begin{itemize}
 \item If some irreducible $S_m$-module is a nonzero component of $\Gamma_m(F_n(A))$, it is also a nonzero component of $\Gamma_m(I)$;

 \item If some irreducible $S_m$-module is a nonzero component of $\Gamma_m(I)$ and vanishes in $\Gamma_m(F_n(A))$, then it is some of the
$(KS_m) \cdot u_{m,i}^{(j)}$ analyzed in the four items b) above.
\end{itemize}

But in the items b) the main argument is that
\[f_{m,i}^{(j)} \in \langle [x_1,x_2]\ldots [x_{n+3},x_{n+4}] \rangle ^T,\]
for an appropriate $i$. In particular, $f_{m,i}^{(j)} \in I$ and then $(KS_m) \cdot u_{m,i}^{(j)}$ vanishes in $\Gamma_m(I)$. Hence
\[\Gamma_m(F_n(A))=\Gamma_m(I)\]
for any  $m$. As a consequence,
\[T(F_n(A))=I,\]
since the characteristic of $K$ is zero.

Let us now show that
\[T(F_n(A))=T(F_{n+1}(A)).\]
Since $F_n(A)$ is a subalgebra of $F_{n+1}(A)$, we have
\[T(A) \subseteq T(F_{n+1}(A))\subseteq T(F_n(A)).\]
From this we have two important remarks:
\begin{itemize}
 \item If some irreducible $S_m$-module is a nonzero component of $\Gamma_m(F_n(A))$, it is also a nonzero component of $\Gamma_m(F_{n+1}(A))$;

 \item If some irreducible $S_m$-module is a nonzero component of $\Gamma_m(F_{n+1}(A))$ and vanishes in $\Gamma_m(F_n(A))$, then it is some of the
$(KS_m) \cdot u_{m,i}^{(j)}$ analyzed in the four items b) above.
\end{itemize}

But in items b) the main argument is that
\[f_{m,i}^{(j)} \in \langle [x_1,x_2]\ldots [x_{n+3},x_{n+4}] \rangle ^T,\]
for an appropriate $i$. In particular, $f_{m,i}^{(j)} \in T(F_{n+1}(A))$ and then $(KS_m) \cdot u_{m,i}^{(j)}$ vanishes in $\Gamma_m(F_{n+1}(A))$. Hence
\[\Gamma_m(F_n(A))=\Gamma_m(F_{n+1}(A))\]
for any $m$. As a consequence we also conclude that
\[T(F_n(A))=T(F_{n+1}(A)),\]
since the characteristic of $K$ is zero.
\end{proof}

It is well-known from the theory (see \cite{BelovRowen}) that any
finitely generated algebra is PI-equivalent to some finite dimensional algebra. For each
$n\geq 2$ we exhibit a finite dimensional algebra PI-equivalent to $F_n(A)$.

Let $E^{(n)}$ denote the finite dimensional unitary Grassmann algebra finitely generated by $e_1,\ldots,e_n$. Denote by $E_0^{(n)}$ and $E_1^{(n)}$ the
even and odd component of $E^{(n)}$, respectively and define the finite dimensional algebra
\[A^{(n)}=
 \left(
 \begin{array}{cc}
 E_0^{(n)}&E^{(n)}\\
 0&E^{(n)}
 \end{array}
\right).
\]

\begin{teo}
If $n\geq 2$ is even, then $F_n(A)$ and $A^{(n)}$ satisfy the same polynomial identities.
\end{teo}

\begin{proof}
Since $A^{(n)}$ is a subalgebra of $A$,
\[[x_1,x_2][x_3,x_4,x_5]\]
is a polynomial identity of $A^{(n)}$. We now show that
\[f=[x_1,x_2][x_3,x_4]\ldots [x_{n+3},x_{n+4}]\]
is also a polynomial identity for $A^{(n)}$: Since $f$ is multilinear, it is enough to substitute variables by elements of the form $ae_{ij} \in A^{(n)}$,
for suitable $a\in E^{(n)}$. Consider the following substitution
\[\Delta =
[a_1e_{i_1j_1},a_2e_{i_2j_2}][a_3e_{i_3j_3},a_4e_{i_4j_4}]\ldots [a_{n+3}e_{i_{n+3}j_{n+3}},a_{n+4}e_{i_{n+4}j_{n+4}}].\]
We have a few cases to analyze:

\

Case 1) $i_1=j_1=i_2=j_2=1$.

In this case we have $\Delta=0$, since $a_1,a_2 \in E_0^{(n)}$.

\

Case 2) $i_1=j_1=1$, $i_2=1$
and $j_2=2$.

In this case, if $\Delta \neq 0$, the only possibility is
\[i_3=j_3=i_4=j_4=\ldots=
i_{n+3}=j_{n+3}=i_{n+4}=j_{n+4}=2.\] Hence,
\[\Delta =
a_1a_2[a_3,a_4]\ldots [a_{n+3},a_{n+4}]e_{12}.\] Since $a_3,a_4,\ldots,a_{n+3},a_{n+4} \in E^{(n)}$ and
\[[x_3,x_4]\ldots [x_{n+3},x_{n+4}] \in T(E^{(n)}),\]
it follows that $\Delta=0$.

\

Case 3) $i_1=1$, $j_1=2$, $i_2=1$ and $j_2=2$.

In this case we have $[a_1e_{12},a_2e_{12}]=0$ and then $\Delta=0$.

\

Case 4) $i_1=1$, $j_1=2$, $i_2=j_2=2$.

The same argument of Case 2 shows that $\Delta=0$.

\

Case 5) $i_1=j_1=i_2=j_2=2$.

In this case, if $\Delta \neq 0$, the only possibility is
\[i_3=j_3=i_4=j_4=\ldots=
i_{n+3}=j_{n+3}=i_{n+4}=j_{n+4}=2.\] Hence,
\[\Delta =
[a_1,a_2][a_3,a_4]\ldots [a_{n+3},a_{n+4}]e_{22}.\] Since
\[[x_3,x_4]\ldots [x_{n+3},x_{n+4}] \in T(E^{(n)}),\]
it follows that $\Delta=0$.

\

We just proved that \[f=[x_1,x_2][x_3,x_4]\ldots [x_{n+3},x_{n+4}]\] is a polynomial identity for $A^{(n)}$ and $T(F_n(A))\subseteq T(A^{(n)})$.

Since $T(F_n(A))\subseteq T(A^{(n)})$, the $S_m$-module $\Gamma_m(A^{(n)})$ is a homomorphic image of $\Gamma_m(F_n(A))$. In order to show that $A^{(n)}$
satisfies the same identities of $F_n(A)$, we need to show that it satisfies no other identities. For that matter it is enough to show that the generators
of the irreducible $S_m$-modules in the decomposition of $\Gamma_m(F_n(A))$ are not identities of $A^{(n)}$, i.e., that the polynomials in items a) in the
four cases in the proof of Theorem \ref{seqcarprdea} and the polynomial $u_m^{(1)}$ are not identities for $A^{(n)}$. In order to do this it is enough to
observe that Lemmas \ref{fm1}, \ref{fm2}, \ref{lemfmn3}, \ref{fm4} and \ref{fm5} also hold for $A^{(n)}$ instead of $F_n(A)$, that the polynomials
$f_{m,i}^{(j)}$ are equivalent to $u_{m,i}^{(j)}$ and apply the same arguments used in items a) in the proof of Theorem \ref{teoremaprincipal}.
\end{proof}

In \cite{KdM} the second author and Koshlulov describe a finite basis for $F_2(M_{1,1})$, the relatively free algebra of rank 2 of the minimal algebra $M_{1,1}$. They also exhibit a finite dimensional algebra PI-equivalent to $F_2(M_{1,1})$. Using the same construction of the above theorem, define for each $k$,
                                                \[M_{1,1}^{(k)}=\left(
                                                    \begin{array}{cc}
                                                      E^{(k)}_0 & E^{(k)}_1 \\
                                                      E^{(k)}_1 & E^{(k)}_0 \\
                                                    \end{array}
                                                \right)
\]
Then using the decomposition of the $S_m$-module $\Gamma_m(F_2(M_{1,1}))$ as a sum of irreducible modules, given in $\cite{KdM}$,
one can verify the following:

\begin{teo}
The algebras $F_2(M_{1,1})$ and $M_{1,1}^{(2)}$ are PI-equivalent.
\end{teo}
It is interesting to investigate if the above result holds for arbitrary $k$.

\begin{prob}
Are the algebras $F_k(M_{1,1})$ and $M_{1,1}^{(k)}$ PI-equivalent, for $k>2$?
\end{prob}

\section{Subvarieties and asymptotic equivalence}

The notion of asymptotic equivalence of varieties (or T-ideals) has been introduced by Kemer in \cite{Kemer90}. If $\mathfrak{U}$ and $\mathfrak{V}$ are
varieties, they are
asymptotically equivalent if $T(\mathfrak{U})$ and
$T(\mathfrak{V})$ satisfy the same proper polynomials from a certain degree on.

In the description of the subvarieties of $\mathfrak{M}_5$ in \cite{DDN}, the authors show that any subvariety of $\mathfrak{M}_5$ is asymptotically equivalent
to the variety generated by one among the following algebras: $K$, $E$, $R_p$ or $R_p\oplus  E$, for a suitable $p$. Here $R_p$ stands for the algebra
\[R_p:=\left\{\left.\left(
                      \begin{array}{cc}
                        \alpha_{11} & \alpha_{12} \\
                        0           & \alpha_{22} \\
                      \end{array}
                    \right)\right| \begin{array}{cc}
                              \alpha_{11}\in E_0, & \alpha_{12}\in \mathcal{C}_pE_0+\mathcal{C}_pE_1 \\
                               & \alpha_{22}\in \mathcal{C}_pE_0+t\mathcal{C}_pE_1
                            \end{array}
 \right\}\]
and $\mathcal{C}_p:=\dfrac{K[t]}{(t^p)}$.

We remark that for each $p\geq 2$,
the polynomial identities of $R_{2p}$ coincide with the polynomial identities of $F_{2p-2}(A)$. In particular, since $R_2$ is PI-equivalent to $UT_2(K)$, the $2\times 2$ upper triangular matrix algebra, we can restate the result of \cite{DDN} as

\begin{coro}
Any subvariety of $\mathfrak{M}_5$ is asymptotically equivalent to the variety generated by one among the following algebras: $UT_2(K)$, $E$, $F_m(A)$ or
$F_m(A)\oplus E$, for a suitable $m$.
\end{coro}

In order to compare the above with the description of subvarieties of other varieties of algebras, we recall that if $R \in var(E)$ ($R$ not necessarily a
unitary algebra). Then $R$ is PI-equivalent to one of the following algebras: $E$, $N$, $C\oplus N$, $E^{(2m)}$, or $E^{(2m)}\oplus N$ , for some $m \geq
1$,  where $N$ is a nilpotent and $C$ is a commutative algebra. Such description can be found in \cite{LaMattina} (see also its review Zbl 1119.16022 in
Zentralblatt f\"ur Mathematik, https://zbmath.org/), and also in bulgarian in the paper \cite{DrVl}.

Let $F_m(E)$ denote the relatively free algebra of rank $m$ of the variety generated by $E$. It is easy to see that for each $m$, $F_{m}(E)$ and $E^{(m)}$
are PI-equivalent. In particular, since we are working with algebras with unit, we have the following description of subvarieties of $var(E)$.

\begin{prop}
Let $R \in var(E)$. Then $R$ is PI-equivalent to one among the following algebras: $K$, $E$, or $F_{2m}(E)$ , for some $m \geq 1$.
\end{prop}

By the above corollaries one expects the relatively free algebras of finite rank to be important in the description of subvarieties of varieties of
associative unitary algebras. This will be investigated in future projects.

\newpage

\section*{Acknowledgments}

The authors would like to thank the anonymous referee for his/her useful suggestions, that much improved the final version of this paper.

\end{document}